\documentclass[11pt,reqno]{amsart}
\usepackage{microtype}

\usepackage[dvipsnames]{xcolor}
\usepackage[colorlinks=true,linkcolor=Maroon,citecolor=OliveGreen,backref]{hyperref}
\usepackage{amsmath, amsfonts,amsthm,amssymb,verbatim,
color,multirow,booktabs,mathdots,bm}
\usepackage[shortlabels]{enumitem}
\setlist[enumerate]{label={(\arabic*)}, wide}

\usepackage[capitalize]{cleveref}
\crefname{equation}{}{}

\usepackage[abbrev, alphabetic]{amsrefs}  

\numberwithin{equation}{section}
\newtheorem{theorem}{Theorem}[section]
\newtheorem{lemma}[theorem]{Lemma}

\newtheorem{proposition}[theorem]{Proposition}
\newtheorem{corollary}[theorem]{Corollary}

\newtheorem*{theorem*}{Theorem}
\newtheorem*{question}{Question}

\newtheorem*{conjecture}{Conjecture}

\theoremstyle{definition}
\newtheorem{definition}[theorem]{Definition}

\theoremstyle{remark}
\newtheorem*{remark}{Remark}
\newtheorem*{remarks}{Remarks}

\def\Z{\mathbf{Z}}
\def\N{\mathbf{N}}
\def\F{\mathbf{F}}

\newcommand\opr[1]{\operatorname{#1}}
\def\Sym{\opr{Sym}}

\def\GL{\opr{GL}}
\def\SL{\opr{SL}}
\def\Sp{\opr{Sp}}
\def\AGL{\opr{AGL}}

\def\nsgp{\trianglelefteq}

\newcommand\br[1]{{\left(#1\right)}}

\newcommand\gen[1]{\left\langle#1\right\rangle}
\newcommand\floor[1]{\left\lfloor#1\right\rfloor}

\renewcommand\bar{\overline}

\def\Gap{\opr{Gap}}
\def\FS{\mathfrak{S}}

\usepackage{todonotes}

\numberwithin{equation}{section}

\author{Sean Eberhard}
\author{Elena Maini}

\address{\parbox{\linewidth}{Mathematics Institute, Zeeman Building, University of Warwick, UK \vspace{0.1cm}}}
\email{firstname.surname@warwick.ac.uk}

\thanks{SE is supported by The Royal Society
(University Research Fellowship, URF\textbackslash{}R1\textbackslash{}221185).
EM is supported by the Warwick Mathematics Institute Centre for Doctoral Training, and gratefully acknowledges funding by the Swinnerton-Dyer scholarship.}

\begin{document}
\title{The growth of residually soluble groups}

\begin{abstract}
    Building on work of Wilson, we show that if $G$ is a finitely generated residually soluble group whose growth function $\gamma$ satisfies $(\log \gamma(n))/ n^{1/4} \to 0$ as $n \to \infty$ then $G$ is virtually nilpotent.
    This shows that Grigorchuk's Gap Conjecture holds for all exponents $\beta < 1/4$ within the class of residually soluble groups (improving Wilson's exponent $1/6$).
    We also discuss stronger versions of the Gap Conjecture.
\end{abstract}

\maketitle

\section{Introduction}
If $G$ is a group generated by a finite set $X$ and $g\in G$, we denote by $\ell_X(g)$ the length of $g$ with respect to $X$, i.e., the smallest integer $n$ such that $g$ can be written as a product of $n$ elements in $X \cup X^{-1}$. For a finite subset $S\subseteq G$, we write $\ell_X(S)=\max\{\ell_X(g): g\in S\}$.
The \emph{growth function} of $X$ is
\begin{equation*}
    \gamma_X(n) = |\{ g\in G:\ell_X(g)\le n \}|.
\end{equation*}

Usually, growth functions are only coarsely compared. Given two functions $f, g : \N \to [0, \infty)$, we write $f \preceq g$ if there is a constant $C > 0$ such that $f(n) \le C g(Cn)$ for all $n > 0$.
This relation defines a preorder on the set of functions.
The equivalence class of the growth function $\gamma_X(n)$ is a group property (i.e., independent of the choice of generating set $X$), and in fact a quasi-isometry invariant.

In 1968, Milnor~\cite{milnor-5603} asked some highly influential questions about the relation between the growth function $\gamma_X(n)$ and the group structure of $G$.
One of his conjectures was that $\gamma_X(n) \preceq n^d$ for some $d$ (if and) only if $G$ is virtually nilpotent;
this is now the celebrated theorem of Gromov~\cite{Gromov}.
Milnor also asked whether the growth of a group is always either bounded by a polynomial or equivalent to $e^n$.
This was proved for soluble groups by Milnor and Wolf~\cites{Milnor,Wolf},
for linear groups by Tits~\cite{Tits},
for elementary amenable groups by Chou~\cite{chou},
but Grigorchuk~\cite{Grig-group-of-intermediate-growth} famously discovered a counterexample, i.e., a group of \emph{intermediate growth}.
In fact, Grigorchuk~\cite{gen-grig} showed that there are a continuum of groups with pairwise incomparable growth functions,
which shows that the landscape of group growth functions is in truth completely dominated by the groups of intermediate growth.

Notwithstanding the existence of groups of intermediate growth, there is undoubtedly a definite gap between polynomial growth and superpolynomial growth.
This follows already from Gromov's theorem and some logical nonsense (see \cite{Gromov}*{p.~71}), but without an explicit function.
Pursuing this idea, Shalom and Tao~\cite{shalom-tao} proved that if $G$ is not virtually nilpotent then $\gamma_X(n) \succeq n^{(\log \log n)^c}$ for some $c > 0$.
This is the greatest lower bound currently known for the growth of non-virtually-nilpotent groups in general.

On the other hand, Grigorchuk's groups all have growth $\succeq \exp(\sqrt n)$.
In particular, the growth of Grigorchuk's first group itself is of the form $\exp(n^{\alpha+o(1)})$, where $\alpha \approx 0.767$ (Erschler--Zheng~\cite{erschler--zheng}),
and this is still the slowest known superpolynomial growth function.
For these reasons among others, Grigorchuk~\cite{grig-ICM} has advocated the following conjecture, known as the \emph{Gap Conjecture}, which we formalize with a free parameter $\beta$.

\begin{conjecture}[$\Gap(\beta)$]
    If $G$ is a group generated by a finite set $X$ and $\gamma_X(n) \prec \exp(n^\beta)$ then $G$ is virtually nilpotent.
\end{conjecture}

We have seen that Milnor's conjecture $\Gap(1)$ is false,
but it does hold within certain large natural classes of groups including soluble groups and linear groups.
``The'' Gap Conjecture is $\Gap(1/2)$,
and it was proved in \cites{Grig-gap-for-residually-nilpotent,lubotzky--mann} that $\Gap(1/2)$ holds for residually nilpotent groups
(such as all of Grigorchuk's groups).
In \cites{wilson2005,Wilson2011}, J.~S.~Wilson considered the wider class of residually soluble groups,
and in the latter paper he proved that $\Gap(1/6)$ holds within this class.

Our main contribution in this paper is to show that Wilson's exponent can be improved to $1/4$.

\begin{theorem}[$\Gap(1/4-\epsilon)$ for residually soluble groups]
\label[theorem]{th-main}
    If $G = \gen X$ is a finitely generated residually soluble group such that
    \[
        \log \gamma_X(n) / n^{1/4} \to 0 \qquad \text{as}~n\to\infty,
    \]
    then $G$ is virtually nilpotent.
    In particular, $\Gap(\beta)$ holds for residually soluble groups for all $\beta < 1/4$.
\end{theorem}

Our proof of \Cref{th-main} builds on Wilson's ideas,
particularly the consideration of self-centralizing chief factors.
We will review everything we need.
The proof splits into two cases.
If the self-centralizing chief factors of $G$ have bounded rank
then $\Gap(1/2)$ holds for $G$.
Otherwise there are self-centralizing chief factors of arbitrarily large rank, and then we can use a generalization of a lemma of Milnor to produce many distinct elements of short length in the chief factor.
To optimize the approach, we use an ad hoc notion of ``modified derived length'', which is a variant of derived length that allows class-2 nilpotent factors at a discount, and we bound the modified derived length of soluble linear groups.

\subsection*{Acknowledgements}

The authors are grateful to both Luca Sabatini and Gareth Tracey for many useful conversations over the past year in connection with the ideas in this paper, especially about the structure of finite soluble groups.
We also wish to record that one of the key ideas in this paper, the generalized Milnor lemma (see \Cref{lem:quantitative-gen-Milnor}), or a close variant anyway, is also critical in forthcoming joint work with Sabatini and Tracey on diameters of permutation groups, and really it originated in that project.

\section{Self-centralizing chief rank}

A key insight of Wilson~\cite{wilson2005} was to consider self-centralizing chief factors.
A chief factor $M / N$ of a finite group $G$ is a minimal normal subgroup of a quotient $G / N$.
If $M / N$ is abelian (e.g, if $G$ is soluble) then $M / N \cong \F_p^n$ for some prime $p$ and integer $n \ge 1$; the \emph{rank} of $M / N$ is $n$.
A chief factor $M / N$ of $G$ is called \emph{self-centralizing} if $M / N = C_{G/N}(M/N)$.
If $G / M$ is soluble, a Frattini argument shows that $M / N$ is complemented in $G / N$, and it follows that $G$ has a self-centralizing chief factor of order $p^n$ if and only if $G$ has a maximal subgroup of index $p^n$.

%

\begin{definition}
    Let $G$ be a finite soluble group.
    The \emph{self-centralizing chief rank}
    (a.k.a.~\emph{non-Frattini chief rank})
    of $G$ is the maximum rank of its self-centralizing chief factors.
    We write $\FS(n)$ for the class of finite soluble groups of self-centralizing chief rank at most $n$.
\end{definition}

\begin{remarks}\leavevmode
    \begin{enumerate}[1.]
        \item A chief factor $M / N$ is called \emph{Frattini} if $M / N \le \Phi(G / N)$. Every self-centralizing chief factor is non-Frattini, and conversely every non-Frattini chief factor is equivalent to a self-centralizing chief factor.
        Refer to \cite{doerk-hawkes}*{A.15}.
        \item Every nontrivial soluble group has positive self-centralizing chief rank, and the class $\FS(1)$ coincides with the class of finite supersoluble groups: see \cite{doerk-hawkes}*{VII~2.2(a, c)}.
    \end{enumerate}
\end{remarks}

If $G$ is a soluble group we denote by $\delta(G)$ its \emph{derived length}: the least integer $k$ such that $G^{(k)} = 1$.
We need the following well-known result of Zassenhaus and Newman (see \cite{newman}). The weak form, due to Zassenhaus, is sufficient in this section, but in the next section we will use the precise bound found by Newman.

\begin{theorem}
    \label[theorem]{thm:zassenhaus-newman}
    Let $\rho(n)$ be the supremum of the derived lengths $\delta(G)$ of the soluble linear groups of degree $n$.
    \begin{enumerate}[(a)]
        \item\label{ZNa} \textup{(Zassenhaus)} $\rho(n) < \infty$.
        \item\label{ZNb} \textup{(Newman)} $\rho(n) < 5 \log_9(n) + 6$.
    \end{enumerate}
\end{theorem}

The following results should be compared with \cite{wilson2005}*{Theorem~1.2}.
Our result is slightly more general. We also avoid both ultrafilter analysis and the Tits alternative.

\begin{theorem}
    If $G \in \FS(n)$ then $G^{(\rho(n))}$ is nilpotent.
\end{theorem}

\begin{proof}
    Let $d = \rho(n)$ and let $N = G^{(d)}$. We will show that $N$ is nilpotent.

    Suppose $\bar G$ is a quotient of $G$ with a minimal normal self-centralizing subgroup $V$ of order $p^m$, where $m \le n$.
    Since $\bar G / V$ embeds in $\GL_m(p)$, it has derived length at most $\rho(n)$, so $\bar N \le V$.
    In particular $[V, \bar N] = 1$.

    Since $N$ acts trivially by conjugation on every self-centralizing chief factor of $G$, it follows from \cite{wilson2005}*{Lemma~2.3} that $N$ acts trivially on every chief factor of $G$.
    Intersecting a chief series for $G$ with $N$ thus gives a central series for $N$, so $N$ is nilpotent.
\end{proof}

\begin{corollary}
    \label[corollary]{cor:resid-FS(n)}
    If $G$ is residually $\FS(n)$ then $G^{(\rho(n))}$ is residually nilpotent.
\end{corollary}

The following basic fact about virtually soluble groups is well-known.

\begin{lemma}
    If $G$ is a group and $N \nsgp G$ such that both $N$ and $G/N$ are virtually soluble, then $G$ itself is virtually soluble.
\end{lemma}
\begin{proof}
    By passing to a finite-index subgroup we may assume $G/N$ is soluble.
    Let $M \nsgp N$ be a soluble normal subgroup of $N$ of minimal index.
    If $g \in G$ then $M M^g \nsgp N$ is again soluble, so we must have $M^g = M$.
    Therefore $M \nsgp G$.
    Replacing $G$ with $G / M$, we may thus assume that $N$ is finite.
    Now $C_G(N)$ is a finite-index soluble subgroup of $G$.
\end{proof}

\begin{theorem}
    \label[theorem]{cor:gap-FS(k)}
    $\Gap(1/2)$ holds for residually $\FS(n)$ groups.
    In fact, if $G = \gen X$ is residually $\FS(n)$ and not virtually nilpotent then $\gamma_X(n) \succeq \exp(n^{1/2})$.
\end{theorem}
\begin{proof}
    Suppose $G$ is residually $\FS(n)$ and of subexponential growth.
    By \Cref{cor:resid-FS(n)}, $G$ has a residually nilpotent normal subgroup $N$ such that $G/N$ is soluble.
    By Milnor's lemma (recalled in \Cref{lem:milnor} below), $N$ is finitely generated.
    From \cites{Grig-gap-for-residually-nilpotent,lubotzky--mann} we conclude that either $G$ has growth $\succeq \exp(n^{1/2})$ or $N$ is virtually nilpotent.
    Since the class of virtually soluble groups is closed under extension,
    in the latter case $G$ itself is virtually soluble.
    Finally, by Milnor--Wolf, $G$ is virtually nilpotent.
\end{proof}

\section{Variations on a lemma of Milnor}

The following fundamental lemma is due to Milnor: see \cite{Milnor}*{Lemma~1} or \cite{mann-book}*{Theorem~5.1}.
Compare also Gromov's lemma (b) in \cite{Gromov}*{Section~4}.
It follows easily from this lemma and induction on derived length that any soluble group of subexponential growth is polycyclic,
which is roughly half of the Milnor--Wolf theorem.

\begin{lemma}[Milnor's lemma]
    \label[lemma]{lem:milnor}
    Let $G$ be a finitely generated group of subexponential growth and assume that $G / N \cong \Z$. Then $N$ is finitely generated.
\end{lemma}

If we assume a stronger growth hypothesis then a more general property holds: every finitely generated subgroup has finitely generated normal closure.
This stronger property does not hold in general for groups of subexponential growth:
see for example~\cite{bartholdi-erschler}*{Theorem~B}.

\begin{lemma}[Generalized Milnor lemma]
    \label[lemma]{lem:soft-generalized-milnor}
    Let $G = \gen X$ be a finitely generated group and assume that $H \le G$ is a finitely generated subgroup whose normal closure $N = \gen{H^G}$ is not finitely generated.
    Then $G$ has growth \[\gamma_X(n) \succeq \exp(n^{1/2}).\]
\end{lemma}
\begin{proof}
    Suppose $H = \gen{Y}$ where $Y$ is finite.
    Let
    \begin{equation}
        \label{eq:Y_i-def}
        Y_i = \{y^g : y \in Y,~\ell_X(g) \le i\}, \qquad H_i = \gen{Y_i} \qquad (i \ge 0).
    \end{equation}
    Then $H_0 \le H_1 \le \cdots$ is an increasing chain of subgroups of $N$. If $H_k = H_{k+1}$ for some $k$ then for $x \in X \cup X^{-1}$ we have
    \[
        H_k^x = \gen{Y_k^x} \le \gen{Y_{k+1}} = H_{k+1} = H_k,
    \]
    so $H_k \nsgp G$, so $H_k = N$.
    Since $N$ is not finitely generated, there is no such $k$.

    Therefore, for each $i > 0$ there is some $y_i \in Y_i \setminus H_{i-1}$.
    Now for all $k \ge 1$ we claim that the $2^k$ elements $y_1^{\epsilon_1} \cdots y_k^{\epsilon_k}$ with $\epsilon_i \in \{0,1\}$ are distinct. Indeed, suppose
    \[
        y_1^{\epsilon_1} \cdots y_k^{\epsilon_k} = y_1^{\delta_1} \cdots y_k^{\delta_k}
    \]
    where $\epsilon,\delta \in \{0,1\}^k$. Then
    \(
        H_{k-1} y_k^{\epsilon_k} = H_{k-1} y_k^{\delta_k},
    \)
    which implies $\epsilon_k = \delta_k$ since $y_k \notin H_{k-1}$.
    Cancelling $y_k^{\epsilon_k}$ and applying induction, it follows that $\epsilon = \delta$.

    Now observe that
    \[
        \ell_X(y_1^{\epsilon_1} \cdots y_k^{\epsilon_k}) \le \sum_{i=1}^k \ell_X(y_i) \le k (L + 2k),
    \]
    where $L = \ell_X(Y)$.
    Thus
    \(
        \gamma_X(kL + 2k^2) \ge 2^k
    \)
    for all $k \ge 1$,
    which implies that $\gamma_X(n) \succeq \exp(n^{1/2})$.
\end{proof}

If we assume an even stronger growth hypothesis then we have a quantative form of the previous lemma.
This is what we will use in the rest of the paper.

\begin{lemma}[Generalized Milnor lemma, quantitative version] \label[lemma]{lem:quantitative-gen-Milnor}
    Let $G = \gen X$ be a finitely generated group whose growth function satisfies
    \begin{equation}
        \label{eq:theta-growth}
        \gamma_X(n) \le \exp(C n^\theta) \qquad(\text{all}~n>0),
    \end{equation}
    for some $\theta < 1/2$ and $C \ge 1$.
    Let $N = \gen{Y^G}$ be a normal subgroup of $G$,
    where $Y$ is finite.
    Then $N = \gen Z$ where $Z$ is finite and
    \[
        \ell_X(Z) \le \ell_X(Y) + C_1 \ell_X(Y)^{\frac{\theta}{1-\theta}}.
    \]
    Here $C_1 = 2 (5C)^{1/(1-2\theta)}$ is a constant depending only on $(C, \theta)$.
\end{lemma}
\begin{proof}
    Define $Y_i$ and $H_i$ as in \eqref{eq:Y_i-def}.
    As before, $H_0 \le H_1 \le \cdots$ is an increasing chain of subgroups,
    and if $H_k = H_{k+1}$ then $H_k = N$.
    Moreover, $H_k = \gen{Y_k}$ where $\ell_X(Y_k) \le \ell_X(Y) + 2k$.
    Thus it suffices to prove that there is such an integer $k$ and to bound it.

    Suppose $H_{i-1} < H_i$ for all $i = 1, \dots, k$.
    For each such $i$ pick $y_i \in Y_i \setminus H_{i-1}$.
    As before, the $2^k$ elements $y_1^{\epsilon_1} \dots y_k^{\epsilon_k}$ with $\epsilon_i \in \{0,1\}$ are distinct,
    and it follows that
    \(
        \gamma_X(kL + 2k^2) \ge 2^k,
    \)
    where $L = \ell_X(Y)$.
    Now from the growth hypothesis \eqref{eq:theta-growth} we get
    \[
        k \le \log_2 \gamma_X(kL + 2k^2) \le 2C (kL + 2k^2)^\theta.
    \]
    If $L \le k$, the above inequality gives $k^{1-2\theta} \le 5 C$.
    Otherwise, it gives $k^{1-\theta} \le 5 C L^\theta$.
    Thus, since $\theta < 1/2$, and since we may obviously assume $L \ge 1$, either way we get $k \le (5C)^{1/(1-2\theta)} L^{\theta / (1-\theta)}$.
    This finishes the proof.
\end{proof}

\begin{corollary}
    \label[corollary]{cor:derived-series-generators}
    Let $G = \gen X$ be a finitely generated group satisfying \eqref{eq:theta-growth} for some $\theta < 1/2$ and $C \ge 1$.
    For each $k \ge 0$, $G^{(k)} = \gen {X_k}$ for some finite set $X_k$ such that
    \[
        \ell_X(X_k) \le C_2 4^k.
    \]
    Here $C_2$ is another constant depending only on $(C, \theta)$.
\end{corollary}
\begin{proof}
    We use induction on $k$.
    To start we may take $X_0 = X$.
    If $k > 0$, by induction we have $G^{(k-1)} = \gen{X_{k-1}}$ for a suitable set $X_{k-1}$.
    Let
    \[
        Y_k = \{[x_1, x_2] : x_1, x_2 \in X_{k-1}\}.
    \]
    Since $G^{(k-1)} / \gen{Y_k^G}$ is abelian, $G^{(k)} = \gen{Y_k^G}$.
    Therefore by the lemma we have $G^{(k)} = \gen{X_k}$ for some finite set $X_k$ with
    \[
        \ell_X(X_k) \le 4 \ell_X(X_{k-1}) + 4 C_1 \ell_X(X_{k-1})^{\theta / (1-\theta)}.
    \]
    This shows that $\ell_X(X_k) \le L_k$, where the sequence $(L_k)$ is defined by $L_0 = 1$ and
    \[
        L_{k+1} = 4 L_k + 4 C_1 L_k^\delta \qquad (k \ge 1),
    \]
    where $\delta = \theta / (1-\theta) < 1$.
    Clearly $L_k \ge 4^k$ for all $k$. Therefore also
    \[
        \frac{L_{k+1}}{4 L_k}
        \le 1 + C_1 L_k^{-1 + \delta}
        \le 1 + C_1 4^{- (1-\delta) k}
        \le \exp(C_1 4^{-(1-\delta) k}),
    \]
    and it follows that
    \[
        \frac{L_k}{4^k} = \prod_{j=0}^{k-1} \frac{L_{j+1}}{4 L_j} \le \exp\br{C_1 \sum_{j=0}^\infty 4^{-(1-\delta)j}},
    \]
    and this proves the claim.
\end{proof}

By combining the previous corollary with the result of the previous section we can easily obtain $\Gap(1/4.16)$ for residually soluble groups, which is only slightly weaker than our main result \Cref{th-main}.
In the next section we will embellish the method to obtain $\Gap(1/4 - \epsilon)$, as claimed in \Cref{th-main}.

\begin{proposition}
    \label[proposition]{prop:4.16}
    $\Gap(1/4.16)$ holds for residually soluble groups.
\end{proposition}
\begin{proof}
    Suppose that $G$ is a non-virtually-nilpotent residually soluble group and $X$ is a finite generating set for $G$ such that $\gamma_X(n) \preceq \exp(n^\beta)$, where $\beta = 1/4.16$.
    Then in particular $X$ satisfies \eqref{eq:theta-growth} with $\theta = 1/3$ (say) and some constant $C$.
    By \Cref{cor:derived-series-generators}, $G^{(k)} = \gen{X_k}$ where $X_k$ is a finite set such that
    \[
        \ell_X(X_k) \le C_2 4^k.
    \]
    In particular $G$ is residually polycyclic.
    Since polycyclic groups are residually finite (see \cite{mann-book}*{Theorem~2.24}), $G$ is residually finite-soluble.

    On the other hand, \Cref{cor:gap-FS(k)} implies that $G$ is not residually $\FS(n)$,
    so there are arbitrarily large integers $n$ such that $G$ has a finite soluble quotient $\bar G = G / N$ with a self-centralizing minimal normal subgroup $V = M / N$ of rank $n$, say $V \cong \F_p^n$.
    Since $V$ is minimal normal and self-centralizing, $V$ must coincide with the last nontrivial term of the derived series of $\bar G$.
    Thus
    \[
        V = \bar G^{(k)},\qquad\text{where}~k = \delta(\bar G / V).
    \]
    Since $\bar G / V$ embeds in $\GL_n(p)$, \Cref{thm:zassenhaus-newman}\ref{ZNb} gives
    \[
        k \le \rho(n) < 5 \log_9(n) + 6.
    \]

    Since $G^{(k)} = \gen {X_k}$, we can find $n$ elements $b_1, \dots, b_n \in X_k$ whose images in $V$ are linearly independent.
    It follows that the $2^n$ elements
    \[
        b_1^{\epsilon_1} \cdots b_n^{\epsilon_n} \qquad (\epsilon_1, \dots, \epsilon_n \in \{0,1\})
    \]
    are distinct, and
    \[
        \ell_X(b_1^{\epsilon_1} \cdots b_n^{\epsilon_n})
        \le C_2 4^k n
        \le C_2 4^{5 \log_9(n) + 6} n
        \le C_3 n^{5 \log_9(4) + 1} < C_3 n^{4.155}.
    \]
    Therefore
    \[
        \gamma_X(C_3 n^{4.155}) \ge 2^n.
    \]
    This is a contradiction to the hypothesis $\gamma_X(n) \preceq \exp(n^\beta)$.
\end{proof}

\section{Modified derived length}

To improve the exponent $1/4.16$ to $1/4$ we have to be creative, since Newman's bound on the derived length of soluble linear groups is sharp.
Our idea is to consider a variant of derived length in which we allow sections to be nilpotent of class at most $2$.
The motivation for this idea is the observation that, if $N = \gen X \nsgp G$, instead of applying \Cref{lem:quantitative-gen-Milnor} to the derived subgroup $N'$, which is normally generated by the set of commutators $[x_1, x_2]$ ($x_1, x_2 \in X$),
we can equally apply it to $\gamma_3(N)$, which is normally generated by the set of triple commutators $[x_1, x_2, x_3]$ ($x_1, x_2, x_3 \in X$).
The word $[x_1, x_2, x_3]$ has length $10$, and $10 < 4^2$, so there is potentially a saving available through this observation.

Let $G$ be a soluble group.
A \emph{class-$2$ series} for $G$ is a subnormal series
\[
    1 = H_k \nsgp H_{k-1} \nsgp \cdots \nsgp H_1 \nsgp H_0 = G \hspace{2cm} (\mathcal S)
\]
such that each section $H_{i-1} / H_i$ is nilpotent of class at most $2$.
Let $A(\mathcal S)$ be the set of indices $i$ such that $H_{i-1} / H_i$ is abelian.
The \emph{cost} of the series $\mathcal S$ is
\[
    \mu_S= \sum_{i\in A(\mathcal{S})} 1 + \sum_{i\in\{1,\dots,k\}\setminus A(\mathcal{S})} \log_4(10).
\]
We write $\mu(G)$ for the minimal cost of a class-$2$ series for $G$, and we call $\mu(G)$ the \emph{modified derived length} of $G$.

Clearly $\mu(G) \le \delta(G)$, but it can be smaller (e.g., for class-2 nilpotent groups).
Also, $\mu$ behaves well with respect to subgroups, quotients, extensions, and direct powers.

\begin{lemma}
    Let $G$ be a soluble group.
    \begin{enumerate}
        \item If $H \le G$ then $\mu(H) \le \mu(G)$.
        \item If $N \nsgp G$ then $\mu(G/N) \le \mu(G)$.
        \item If $N \nsgp G$ then $\mu(G) \le \mu(N) + \mu(G/N)$.
        \item If $n \ge 1$ then $\mu(G^n) = \mu(G)$.
    \end{enumerate}
\end{lemma}

\begin{corollary}
    Let $A$ and $B$ be soluble groups. Suppose $B$ acts transitively on a set $\Omega$ and $W = A \wr_\Omega B = A^\Omega \rtimes B$. Then $\mu(W) \le \mu(A) + \mu(B)$.
\end{corollary}

\begin{remark}
    It is not generally true that $\mu(G_1 \times G_2) = \max(\mu(G_1), \mu(G_2))$. For example, this fails for $G_1 = \SL_2(3)$ and $G_2 = \F_3^2 \rtimes Q_8$.
\end{remark}

In the proof of \Cref{prop:4.16} we used Newman's bound \Cref{thm:zassenhaus-newman}\ref{ZNb} on the derived length of a soluble linear group.
We will prove a stronger bound for the modified derived length.
In order to prove such a bound, we will need Suprunenko's structure theorem for soluble primitive linear groups: see~\cites{Suprunenko1, Suprunenko2} (the statement below follows \cite{Wilson2011}*{Lemma 2.7} closely).

\begin{theorem}[Suprunenko]
\label[theorem]{thm:suprunenko}
    Let $G$ be a maximal primitive soluble subgroup of $\GL_n(p)$.
    Then $G$ has a unique maximal abelian normal subgroup $A$, which is cyclic of order $p^l - 1$ for some divisor $l$ of $n$.
    Moreover, $C = C_G(A)$ embeds in $\GL_r(p^l)$ where $r = n / l$ and $G/C$ is cyclic of order $l_1$ for some divisor $l_1$ of $l$.
    If $A=C$, then $l_1=l$.
    Otherwise, $r > 1$ and there is an integer $u>1$ that divides $r$ and all of whose prime divisors divide $p^l - 1$, such that $C$ has the following structure:
    \begin{enumerate}[(i)]
        \item $C/A$ has a unique maximal abelian normal subgroup $B/A$; the group $B/A$ has elementary abelian Sylow subgroups and order $u^2$;
        \item $C/B$ embeds in $\prod_{i=1}^s $ $ \Sp_{2 k_i}(q_i)$, where $u=\prod_{i=1}^s {q_i}^{k_i}$ is the prime factorization of $u$, and the projection of $C/B$ in each symplectic group is completely reducible.
    \end{enumerate}
\end{theorem}

\begin{remark}
    $A = Z(C)$ and $B$ is nilpotent of class $2$.
\end{remark}

Essentially by applying Suprunenko's theorem, Newman obtained the following precise bound for the derived length of a completely reducible linear group.

\begin{theorem}[Newman]
    \label[theorem]{lem:newman-sigma}
    Let $G \le \GL_n(p)$ be soluble and completely reducible.
    Then $\delta(G) \le \sigma(n)$, where
    \[
        \sigma(n) =
        \begin{cases}
            1 &\text{if}~n=1,\\
            4 &\text{if}~n=2,\\
            5 &\text{if}~n \in [3,5],\\
            6 &\text{if}~n \in [6,7],\\
        \end{cases}
    \]
    and in general
    \[
        \sigma(n) \le \floor{5 \log_9(n) + C}
    \]
    where $C = 8 - 15 \log_9(2) \approx 3.268$.
\end{theorem}
\begin{proof}
    See \cite{newman}*{Theorem~C$_S$ and Appendix}.
\end{proof}

We will prove a stronger bound for $\mu(G)$ when $G$ is soluble and irreducible, using a similar approach. To begin, we need to bound $\mu(G)$ and $\delta(G)$ in several special cases.

\begin{lemma}\label[lemma]{lem:small-cases}
    Let $G \le \GL_n(p)$ be soluble and irreducible.
    \begin{enumerate}
        \item\label{case-2} If $n=2$ then $\mu(G) \le 2 + \log_4(10)$.
        \item\label{case-3} If $n=3$ then $\mu(G)\le 1+2\log_4(10)$.
        \item\label{case-4} If $n=4$ then $\mu(G)\le 3 + \log_4(10)$.
    \end{enumerate}
    Furthermore, if $p = 2$ then the following bounds hold.
    \begin{enumerate}[start=4]
        \item\label{case-prime-2} If $n$ is prime then $\delta(G)\le 2$.
        \item\label{case-4-2} If $n = 4$ then $\delta(G)\le 3$.
        \item\label{case-6-2} If $n = 6$ then $\delta(G) \le 6$ and $\mu(G)\le 2 + 2\log_4(10)$.
        \item\label{case-8-2} If $n = 8$ then $\delta(G)\le 5$.
        \item\label{case-9-2} If $n = 9$ then $\delta(G) \le 5$.
        \item\label{case-10-2} If $n = 10$ then $\delta(G) \le 4$.
    \end{enumerate}
\end{lemma}

\begin{proof}
    Without loss of generality $G$ is a maximal soluble irreducible subgroup of $\GL_n(p)$. If $G$ is primitive then \Cref{thm:suprunenko} applies and we use the notation from that theorem.
    In this case we will use the following observation several times: if there does not exist a factorization $n = rl$ such that $\gcd(r, p^l - 1) > 1$, for example if $\gcd(n, p^n - 1) = 1$, then necessarily $A = C$ and $\delta(G) \le 2$.

    \ref{case-2}
    Assume $n = 2$.

    \emph{Imprimitive case.}
    $G = \GL_1(p) \wr S_2$, so $\delta(G)\le 2$.

    \emph{Primitive case.}
    If $A=C$ then $\delta(G)\le 2$.
    If $A\neq C$, then $G=C$, $u=2$, and $C/B\le \GL_2(2) \cong S_3$, so $\mu(G) \le 2 + \log_4(10)$.

    \ref{case-3}
    Assume $n = 3$.

    \emph{Imprimitive case.} $G = \GL_1(p) \wr S_3$, $\delta(G)\le 3$.

    \emph{Primitive case.}
    If $A = C$ then $\delta(G)\le 2$.
    If $A\neq C$ then $l=1$, $u=3$, $G = C$, and $C/B\le \SL_2(3) \cong Q_8 \cdot 3$,
    so $\mu(C / B) \le 1 + \log_4(10)$ and $\delta(C / B) \le 3$,
    which gives $\mu(G) \le 1 + 2 \log_4(10)$.

    We defer the proof of \ref{case-4} until after the proof of \ref{case-prime-2} and \ref{case-4-2}.

    \ref{case-prime-2}
    Assume $n$ is prime and $p = 2$.

    \emph{Imprimitive case.} $G$ is isomorphic to a maximal soluble transitive subgroup of $S_n$, and hence $G \cong \AGL_1(n)$ (see \cite{DM}*{Exercise~3.5.1}), so $\delta(G) \le 2$.

    \emph{Primitive case.} $A \cong C_{2^l - 1}$ for some $l$, and since $A$ cannot be trivial we must have $l > 1$, so $l = n$.
    Thus $A = C$ and $\delta(G)\le 2$.

    \ref{case-4-2}
    Assume $(n, p) = (4, 2)$.

    \emph{Imprimitive case.} Either $G \cong S_4$ or $G \cong \GL_2(2) \wr S_2 \cong S_3 \wr S_2$, so $\delta(G) \le 3$.

    \emph{Primitive case.}
    Since $\gcd(4, 2^4 - 1) = 1$ we must have $A = C$, so $\delta(G) \le 2$.

    \ref{case-4}
    Assume $n = 4$.

    \emph{Imprimitive case.} $G \le \GL_d(p) \wr \Sym(4/d)$ for $d \in \{1, 2\}$.
    If $d = 1$ then $G = \GL_1(p) \wr S_4$ and $\delta(G) \le 4$.
    If $d = 2$ then $G = L \wr S_2$ for some soluble irreducible $L \le \GL_2(p)$, so $\mu(G) \le 3 + \log_4(10)$ by \ref{case-2}.

    \emph{Primitive case.}
    If $A = C$ then $\delta(G) \le 2$.
    Otherwise, $r > 1$ and $u \in \{2, 4\}$.
    If $u = 2$ then $C / B \le \Sp_2(2) \cong S_3$
    and $\mu(G) \le 3 + \log_4(10)$.
    If $u = 4$ then $l = 1$, $G = C$, and $C / B \le \Sp_4(2)$.
    If $C / B$ is reducible then $\delta(C/B)\le 2$ by \ref{case-prime-2}.
    If $C / B$ is irreducible then $\delta(C / B) \le 3$ by \ref{case-4-2}.
    Thus $\mu(G) \le 3 + \log_4(10)$ in all cases.

    In the remaining cases we use the following basic observation. If $G$ is a maximal soluble transitive subgroup of $S_n$ then $G$ is isomorphic to a wreath product $P_1 \wr \cdots \wr P_k$ where each $P_i$ is maximal primitive soluble of some prime-power degree $q_i$ such that $n = q_1\cdots q_k$. Moreover if $q_i \le 4$ then $P_i = S_{d_i}$, and more generally if $q_i = p_i^{d_i}$ with $p_i$ prime then $P_i \le \AGL_{d_i}(p_i)$ (see \cite{DM}*{Theorem~4.7A}).

    \ref{case-6-2}
    Assume $(n, p) = (6, 2)$. The bound on $\delta(G)$ follows from \Cref{lem:newman-sigma}, so we just need to prove the bound on $\mu(G)$.

    \emph{Imprimitive case.} $G \le \GL_d(2) \wr \Sym(6/d)$ for some $d \in \{1, 2, 3\}$.
    If $d = 1$ then $G \le S_6$. Applying the observation above, $G \cong S_2 \wr S_3$ or $G \cong S_3 \wr S_2$, so $\delta(G) = 3$.
    If $d = 2$ then $G = \GL_2(2) \wr S_3 \cong S_3 \wr S_3$, so $\delta(G) = 4$.
    If $d = 3$ then $G = L \wr S_2$ for some soluble irreducible subgroup $L \le \GL_3(2)$, and by \ref{case-prime-2} we have $\delta(G) \le 3$.

    \emph{Primitive case.}
    If $A = C$ then $\delta(G) \le 2$.
    Otherwise, $u \mid 6$ and the prime divisors of $u$ divide $2^l-1$, so $u=3$ and $C/B\le \SL_2(3) \cong Q_8 \cdot 3$, so $\mu(C / B) \le 1 + \log_4(10)$.
    Since $B$ is nilpotent of class $2$ we get $\mu(G) \le 2 + 2 \log_4(10)$.

    \ref{case-8-2}
    Assume $(n, p) = (8, 2)$.

    \emph{Imprimitive case.} $G \le \GL_d(2) \wr \Sym(8/d)$ for some $d \in \{1, 2, 4\}$.
    If $d = 1$ then $G \le S_8$ and $G$ is either $S_4 \wr S_2$, $S_2 \wr S_4$, or else $G \le \AGL_3(2)$.
    In the last case $G \cong \F_2^3 \rtimes L$ for an irreducible soluble subgroup $L \le \GL_3(2)$, and $\delta(L) \le 2$ by \ref{case-prime-2}.
    Thus $\delta(G) \le 4$ in these cases.
    If $d = 2$ then $G = \GL_2(2) \wr S_4 \cong S_3 \wr S_4$ and $\delta(G) = 5$.
    If $d = 4$ then $G = L \wr S_2$ for some soluble irreducible $L \le \GL_4(2)$, and $\delta(L) \le 3$ by \ref{case-4-2}, so $\delta(G) \le 4$.

    \emph{Primitive case.} Since $\gcd(8, 2^8 - 1) = 1$ we must have $A = C$ and $\delta(G) \le 2$.

    \ref{case-9-2}
    Assume $(n, p) = (9, 2)$.

    \emph{Imprimitive case.} $G \le \GL_d(2) \wr \Sym(9/d)$ for some $d \in \{1, 3\}$.
    If $d = 1$ then $G \le S_9$ and $G$ is either $S_3 \wr S_3$ or $\AGL_2(3)$, so $\delta(G) \le 5$.
    If $d = 3$ then $G = L \wr S_3$ for some soluble irreducible subgroup $L \le \GL_3(2)$, so by \ref{case-prime-2} we have $\delta(G) \le 4$.

    \emph{Primitive case.} Since $\gcd(9, 2^9-1) = 1$, $A = C$ and $\delta(G) \le 2$.

    \ref{case-10-2}
    Assume $(n, p) = (10, 2)$.

    \emph{Imprimitive case.}
    $G \le \GL_d(2) \wr \Sym(10/d)$ for some $d \in \{1, 2, 5\}$.
    If $d = 1$ then $G \le S_{10}$ and $G \cong \AGL_1(5) \wr S_2$ or $G \cong S_2 \wr \AGL_1(5)$, so $\delta(G) = 3$.
    If $d = 2$ then $G = \GL_2(2) \wr \AGL_1(5) \cong S_3 \wr \AGL_1(5)$ and $\delta(G) = 4$.
    If $d = 5$ then $G = L \wr S_2$ for some irreducible $L \le \GL_5(2)$ and $\delta(L) \le 2$ by \ref{case-prime-2}, so $\delta(G) \le 3$.

    \emph{Primitive case.} Since $\gcd(10, 2^{10} - 1) = 1$, $A = C$ and $\delta(G) \le 2$.
\end{proof}

\begin{theorem} \label[theorem]{prop-modified derived length}
    Let $n \ge 1$.
    \begin{enumerate}[(i)]
        \item If $G$ is a transitive soluble subgroup of $S_n$ then
        \(\mu(G) \le 3 \log_4(n).\)
        \item If $G$ is an irreducible soluble subgroup of $\GL_n(p)$ then \[\mu(G) \le 3\log_4(n) + K,\]
        where $K = -5/2 + 3 \log_4(10) \approx 2.48$.
    \end{enumerate}
\end{theorem}

\begin{proof}
    We prove both statements together by induction on $n$.
    More precisely, we will assume that both \emph{(i)} and \emph{(ii)} hold for all positive integers $m < n$,
    and additionally in the proof of \emph{(ii)} we will assume that \emph{(i)} holds for $m = n$.

    \emph{(i)}
    Let $G \le S_n$ be soluble and transitive. If $1 \le n \le 4$ then
    \[
        \mu(G) \le \mu(S_n) \le n-1 \le 3 \log_4(n),
    \]
    so we may assume $n \ge 5$.
    If $G$ is imprimitive then $G\le L \wr M$ for some soluble transitive subgroups $L\le \Sym(n/t)$ and $M\le \Sym(t)$ with $t \mid n$ and $1 < t < n$, so by induction
    \[
        \mu(G) \le \mu(L) + \mu(M) \le 3\log_4(n/t)+ 3\log_4(t) = 3\log_4(n).
    \]
    If $G$ is primitive then $n = p^d$ and $G \cong V \rtimes H$ with $V = \F_p^d$ and $H \le \GL_d(p)$ irreducible.
    Since $d \le \log_2(n) < n$, by induction
    \[
        \mu(G)\le 1 + 3\log_4(d) + K,
    \]
    and this gives $\mu(G) < 3 \log_4(n)$ unless $n \in \{8, 9, 16\}$.
    In these cases we refer to \Cref{lem:small-cases}.
    If $n \in \{8, 16\}$ then $H \le \GL_d(2)$ with $d \in \{3, 4\}$, and \Cref{lem:small-cases}\ref{case-prime-2}\ref{case-4-2} give $\delta(H) \le 3$, so
    \[
        \mu(G) \le \delta(G) \le  4 < 3 \log_4(n).
    \]
    If $n = 9$ then $H \le \GL_2(3)$, so
    \[
    \mu(G) \le 1 + \mu(\GL_2(3)) = 3 + \log_4(10) < 3 \log_4(n).
    \]

    \emph{(ii)}
    Let $p$ be a prime and let $G\le \GL_n(p)$ be soluble and irreducible.
    If $G$ is imprimitive then $G\le L \wr M$ for some soluble irreducible subgroup $L$ of $\GL_{n/t}(p)$ and some soluble transitive subgroup $M$ of $\Sym(t)$,
    where $t \mid n$ and $1 < t \le n$,
    so by induction we have
    \[
        \mu(G) \le \mu(L) + \mu(M) \le 3\log_4(n/t) + K + 3 \log_4(t) = 3 \log_4(n) + K.
    \]
    Thus we may assume $G$ is primitive, and without loss of generality we may take $G$ to be maximal primitive soluble and apply Suprunenko's~\Cref{thm:suprunenko}.
    If $A = C$ then $\mu(G) \le \delta(G) \le 2$, which is enough to conclude since $K \ge 2$, so we may assume $A \neq C$.
    Since $B$ is nilpotent of class $2$, $\mu(B) \le \log_4(10)$.
    Now $C / B \le \prod_{i=1}^s \Sp_{2k_i}(q_i)$ where $u = \prod_{i=1}^s q_i^{k_i} \mid n / l$,
    and the projection of $C / B$ to each factor is completely reducible,
    so $\delta(C / B) \le \sigma(2k)$ where $k = \max\{k_1, \dots, k_s\}$.
    Thus
    \begin{equation}
        \label{eq:mu-bound}
        \mu(G) \le 1 + \delta(C / B) + \mu(B) \le \sigma(2k) + 1 + \log_4(10).
    \end{equation}

    Now, $k \le \log_2(n)$, so \Cref{lem:newman-sigma} gives
    \begin{align*}
        \sigma(2k) + 1 + \log_4(10) \le 5 \log_9(2\log_2(n)) + 6
        = 3 \log_4(n) + K + f(\log_2(n)),
    \end{align*}
    where $f(x) = 5 \log_9(2 x) + 6 - 3x/2 - K$.
    We have $f'(x) < 0$ for $x \ge 2$ and $f(7) < 0$, so $f(x) < 0$ for all $x \ge 7$.
    Thus it remains to consider the cases $n \le 2^7 = 128$.
    For each such $n$, let $k$ be the maximum exponent in the prime factorization of $n$ and compare the bound \eqref{eq:mu-bound} with $3 \log_4(n) + K$.
    This analysis gives directly $\mu(G) \le 3 \log_4(n) + K$ in all cases except $n \le 6$ and $n \in \{8, 9, 16, 32\}$, which we must treat specially,
    referring to \Cref{lem:small-cases}.

    \emph{Case $n = 1$.}
    Trivial.

    \emph{Cases $n = 2, 3$.}
    See \Cref{lem:small-cases}\ref{case-2}\ref{case-3}.

    \emph{Cases $n = 4, 5, 6$.}
    The bound $\mu(G) \le \delta(G) \le \sigma(n)$ from Newman's \Cref{lem:newman-sigma} is sufficient.

    \emph{Case $n = 8 = 2^3$.}
    If $l > 1$ then $C / B \le \Sp_4(2)$ and by \Cref{lem:small-cases}\ref{case-prime-2}\ref{case-4-2} we have $\delta(C / B) \le 3$, so
    \[
        \mu(G) \le 1 + \delta(C / B) + \mu(B) \le 4 + \log_4(10)
        < 3 \log_4(n) + K.
    \]
    If $l = 1$ then $G = C$ and $C / B \le \Sp_6(2)$.
    If $C / B$ is reducible then by \Cref{lem:small-cases}\ref{case-prime-2}\ref{case-4-2} we have $\mu(C/B) \le \delta(C / B) \le 3$.
    If $C / B$ is irreducible then by \Cref{lem:small-cases}\ref{case-6-2} we have $\mu(C / B) \le 2 + 2 \log_4(10)$.
    Thus in all cases
    \[
        \mu(G) \le \mu(C / B) + \mu(B) \le 2 + 3 \log_4(10) = 3 \log_4(n) + K.
    \]

    \begin{remark}
    This case is the source of the constant $K$.
    \end{remark}

    \emph{Case $n = 9 = 3^2$.}
    If $l > 1$ then $C / B \le \Sp_2(3) \cong Q_8 \cdot 3$, so $\mu(C / B) \le 1 + \log_4(10)$ and
    \[
        \mu(G) \le 1 + \mu(C / B) + \mu(B) \le 2 + 2 \log_4(10)
        < 3 \log_4(n) + K.
    \]
    If $l = 1$ then $C / B \le \Sp_4(3)$, so, by \Cref{lem:newman-sigma}, $\mu(C / B) \le \sigma(4) = 5$ and
    \[
        \mu(G) \le \mu(C / B) + \mu(B) \le 5 + \log_4(10)
        < 3 \log_4(n) + K.
    \]

    \emph{Case $n = 16 = 2^4$.}
    If $l = 1$ then $G = C$ and $C / B \le \Sp_8(2)$.
    By \Cref{lem:small-cases}\ref{case-prime-2}--\ref{case-8-2} we have $\delta(C / B) \le 6$, so
    \[
        \mu(G) \le \delta(C / B) + \mu(B) \le 6 + \log_4(10) < 3 \log_4(n) + K.
    \]
    If $l > 1$ then $C / B \le \Sp_6(2)$.
    If $C / B$ is reducible then $\delta(C / B) \le 3$ and
    \[
        \mu(G) \le 1 + \delta(C / B) + \mu(B) \le 4 + \log_4(10) < 3 \log_4(n) + K.
    \]
    Finally, if $C / B$ is an irreducible subgroup of $\Sp_6(2)$ then, by \Cref{lem:small-cases}\ref{case-6-2},
    \[
        \mu(G) \le 1 + \mu(C / B) + \mu(B) \le 3 + 3\log_4(10) < 3 \log_4(n) + K.
    \]

    \emph{Case $n = 32 = 2^5$.}
    In this final case $C / B$ is a completely reducible subgroup of $\Sp_{10}(2)$, so $\delta(C / B) \le 6$ by \Cref{lem:small-cases}\ref{case-prime-2}--\ref{case-10-2} and
    \[
        \mu(G) \le 1 + \delta(C / B) + \mu(B) \le 7 + \log_4(10)
        < 3 \log_4(n) + K.
    \]
    This completes the proof.
\end{proof}

Finally we can prove \Cref{th-main}.

\begin{proof}[Proof of \Cref{th-main}]
    The proof is mostly the same as the proof of \Cref{prop:4.16}.
    Suppose that $G = \gen X$ is a non-virtually-nilpotent residually soluble group such that $\log \gamma_X(n) / n^{1/4} \to 0$ as $n \to \infty$.
    Then in particular $G$ satisfies \eqref{eq:theta-growth} with $\theta = 1/3$.
    Exactly as in the proof of \Cref{prop:4.16}, there are arbitrarily large integers $n$ such that $G$ has a finite soluble quotient $\bar G = G / N$ with a self-centralizing minimal normal subgroup $V = M / N \cong \F_p^n$.

    Let
    \begin{equation*}
    \hspace{2 cm}\overline{G}=H_0 \ge H_1 \ge H_2 \ge \dots \ge H_k=1 \hspace{2 cm} (\mathcal{S})
\end{equation*}
    be a class-$2$ series for $\bar G$ of cost $\mu(\bar G)$.
    We may assume $H_i = H_{i-1}'$ or $H_i = \gamma_3(H_{i-1})$ for each $i = 1, \dots, k$.
    In particular $H_i \nsgp \bar G$ for each $i$.
    Since $\bar G / V$ embeds in $\GL_n(p)$, \Cref{prop-modified derived length} gives $\mu(\bar G) \le 3 \log_4(n) + 4$.
    Since $V$ is minimal normal and self-centralizing, $V = H_{k-1}$.
    Now if $H_{i-1} = \gen{X_{i-1}}$ then $H_i$ is normally generated by a set $Y_i$ where $\ell_X(Y_i) \le a_i \ell_X(X_{i-1})$ and $a_i = 4$ if $i \in A(\mathcal S)$ and $a_i = 10$ otherwise.
    Thus \Cref{lem:quantitative-gen-Milnor} implies that each $H_i$ is generated by a set $X_i$ of length $\ell_X(X_i) \le L_i$, where $L_0 = 1$ and
    \[
        L_i = a_i L_{i-1} + C L_{i-1}^{1/2} \qquad(1 \le i \le k),
    \]
    where $C$ is a constant.
    Clearly $L_i \ge 4^i$ for all $i$, so
    \[
        \frac{L_i}{a_i L_{i-1}} \le 1 + C 2^{1-i},
    \]
    and it follows easily that
    \(
        L_i \le C' a_1 \cdots a_{i+1}
    \)
    for some constant $C'$.
    In particular,
    \[
        L_{k-1} \le C' a_1 \cdots a_k = C' 4^{\mu(\bar G)} \le C'' n^3
    \]

    We have established that some $n$ elements $b_1, \dots, b_n$ of length at most $C'' n^3$ project to linearly independent elements of $V$. The products
    \[
        b_1^{\epsilon_1} \cdots b_n^{\epsilon_n} \qquad (\epsilon_1, \dots, \epsilon_n \in \{0,1\})
    \]
    are all distinct, and each has length at most $C'' n^4$.
    Thus
    \[
        \gamma_X( C'' n^4 ) \ge 2^n.
    \]
    This is a contradiction to the hypothesis $\log \gamma_X(n) / n^{1/4} \to 0$.
\end{proof}

\begin{remark}
    As to whether the exponent $1/4$ can be improved further, an important example is the pro-finite-soluble group
    \[
        \mathcal G = \lim_{\longleftarrow} S_4 \wr \cdots \wr S_4,
    \]
    which is the automorphism group of an infinite rooted $4$-regular tree.
    Although this group is not topologically finitely generated, it has a rich variety of finitely generated subgroups.
    \Cref{th-main} applies to any such subgroup $G = \gen X$, so we have established that $\gamma_X(n) > \exp(cn^{1/4})$ infinitely often for some constant $c > 0$.
    Can the exponent $1/4$ be improved in this case?
\end{remark}

\section{Stronger gap conjectures}

In \cite{grig-ICM}, Grigorchuk formulated the following stronger version of the gap conjecture.

\begin{conjecture}[$\Gap^*(\beta)$]
    Let $G$ be a group generated by a finite set $X$ and assume $G$ is not virtually nilpotent. Then \(\gamma_X(n) \succeq \exp(n^\beta).\)
\end{conjecture}

$\Gap^*(1/2)$ holds for residually nilpotent groups by the proof of \cites{Grig-gap-for-residually-nilpotent,lubotzky--mann}. \Cref{cor:gap-FS(k)} extends $\Gap^*(1/2)$ to residually $\FS(n)$ groups.
In particular $\Gap^*(1/2)$ holds for residually supersoluble groups, because supersoluble implies residually finite-supersoluble.
This was already asserted in \cite{grig-ICM}*{Theorem~5.6} without full explanation.

The generalized Milnor lemma (\Cref{lem:soft-generalized-milnor}) ought to be useful for $\Gap^*(\beta)$, as it shows that any group having a finitely generated subgroup with non-finitely-generated normal closure has growth $\succeq \exp(n^{1/2})$.

Wilson's results do not give $\Gap^*(\beta)$ for residually soluble groups for any $\beta > 0$ (contrary to the claim in \cite{grig-ICM}*{Theorem~5.7}),
and neither have we succeeded in proving it, though our methods come close.
A positive answer to the following question would help (compare \Cref{cor:derived-series-generators}).

\begin{question}
    Let $G = \gen X$ be a finitely generated group of subexponential growth.
    Do there exist generating sets $X_k$ for $G^{(k)}$ for all $k \ge 0$ such that $\ell_X(X_k)$ grows at most exponentially as a function of $k$?
\end{question}

Modifying the order of quantifiers one more time, we can formulate an even stronger version of the gap conjecture, which we label $\Gap^{**}(\beta)$.

\begin{conjecture}[$\Gap^{**}(\beta)$]
    There is a function $f(n) \succeq \exp(n^\beta)$ such that the following holds.
    Let $G$ be a group generated by a finite set $X$ and assume $G$ is not virtually nilpotent. Then $\gamma_X(n) \ge f(n)$ for all $n$.
\end{conjecture}

This ``uniform'' version of the gap conjecture has not yet been established even for residually nilpotent groups, though remarkably the method of \cites{Grig-gap-for-residually-nilpotent,lubotzky--mann} does give it for \emph{nonlinear} residually nilpotent groups, and one may take $f(n)$ to be the partition function $p(n)$, which famously satisfies
$\log p(n) \sim \pi \sqrt{2n/3}$
(Hardy--Ramanujan).

The theorem of Shalom--Tao (see \cite{shalom-tao}*{Corollary~1.10}) has this kind of strong uniformity: their result is that there is some universal constant $c > 0$ such that every non-virtually-nilpotent finitely generated group $G = \gen {X}$ satisfies
\[
    \gamma_X(n) \ge n^{c (\log \log n)^c} \qquad (n > 1/c).
\]
For another result with this remarkable kind of uniformity, see the recent paper of Lyons--Mann--Tointon--Tessera~\cite{LMTT}.

\bibliography{refs}

\end{document}